\documentclass[12pt]{article}
\usepackage{amsmath,amssymb,amsthm,comment,cite}

\newtheorem{theorem}{Theorem}[section]
\newtheorem{thmy}{Theorem}

\newtheorem{lemma}[theorem]{Lemma}

\newcommand{\dd}{\displaystyle }

\def\barr{\begin{array}}
\def\earr{\end{array}}

\title{On a function related to the Chermak-Delgado measure of a finite group}
\author{Marius T\u arn\u auceanu}
\date{May 10, 2026}

\begin{document}

\maketitle

\begin{abstract}
In this note, we give some conditions for commutativity and (super)solvability of a finite group $G$ based on a function related to the Chermak-Delgado measure $m(G)$. The possible values of this function are also investigated.
\end{abstract}

{\small
\noindent
{\bf MSC2020\,:} Primary 20D30; Secondary 20D60, 20D99.

\noindent
{\bf Key words\,:} Chermak-Delgado measure, abelian group, (super)solvable group.}

\section{Introduction}

In the last years, there has been growing interest in detecting structural properties of a finite group $G$ through the study of some functions related to elements or to subgroups of $G$ (see for example \cite{5}). In what follows, we introduce and study the function
\begin{equation}
f(G)=\dd\frac{m(G)}{|G|^2}\,,\nonumber
\end{equation}where $m(G)$ denotes the Chermak-Delgado measure of $G$. Recall that the \textit{Chermak-Delgado measure} of a subgroup $H$ of $G$ is defined by
\begin{equation}
m_G(H)=|H||C_G(H)|\nonumber
\end{equation}and
\begin{equation}
m(G)={\rm max}\{m_G(H)\mid H\leq G\}.\nonumber
\end{equation}Also, the set 
\begin{equation}
{\cal CD}(G)=\{H\leq G\mid m_G(H)=m(G)\}\nonumber
\end{equation}\newpage\noindent forms a modular self-dual sublattice of the subgroup lattice $L(G)$ of $G$, which is called the \textit{Chermak-Delgado lattice} of $G$. It was first introduced by Chermak and Delgado \cite{3}, and revisited by Isaacs \cite{6}. Some basic properties of the Chermak-Delgado lattices that are used in our note are as follows:
\begin{itemize}
\item[$\cdot$] if $H\in {\cal CD}(G)$, then $C_G(H)\in {\cal CD}(G)$ and $C_G(C_G(H))=H$;
\item[$\cdot$] the minimum subgroup $M(G)$ of ${\cal CD}(G)$ (called the \textit{Chermak-Delgado subgroup} of $G$) is characteristic, abelian, and $Z(G)\subseteq M(G)$;
\item[$\cdot$] the maximum subgroup $M$ of ${\cal CD}(G)$ is characteristic and ${\cal CD}(M)={\cal CD}(G)$.
\end{itemize}

Clearly, we have $f(G)\in(0,1]$ for any finite group $G$, and $f(G)=1$ if and only if $G$ is abelian. It is also clear that the function $f$ is multiplicative, i.e. 
\begin{equation}
f(G_1\times G_2)=f(G_1)f(G_2)
\end{equation}for all finite groups $G_1$, $G_2$ of coprime orders, and that it satisfies:
\begin{itemize}
\item[{\rm a)}]$f(G)=\dd\frac{1}{[G:M(G)][G:M]}$\,;
\item[{\rm b)}]$f(G)=\dd\frac{f(M)}{[G:M]^2}$\,;\hspace{88.5mm}(2)
\item[{\rm c)}]$f(G)\geq\dd\frac{1}{[G:Z(G)]}$\,.
\end{itemize}

Using the function $f$, we are able to provide conditions of commutativity and (super)solvability of $G$. 

\begin{theorem}
Let $G$ be a finite group. The following hold:
\begin{itemize}
\item[{\rm 1.}] If $f(G)>\frac{1}{4}$\,, then $G$ is abelian. Moreover, we have $f(G)=\frac{1}{4}$ if and only if $G$ has a self-centralizing subgroup of index $2$.
\item[{\rm 2.}] If $f(G)>\frac{1}{9}$\,, then $G$ is supersolvable. Moreover, if $f(G)=\frac{1}{9}$\,, then $G$ has a self-centralizing subgroup of index $3$.
\item[{\rm 3.}] If $f(G)>\frac{1}{60}$\,, then $G$ is solvable. Moreover, we have $f(G)=\frac{1}{60}$ if and only if $G/Z(G)\cong A_5$.
\end{itemize}
\end{theorem}\newpage

Note that $D_8$, $S_3$ provide examples for item 1, $A_4$ for item 2 and $A_5$, ${\rm SL}(2,5)$ for item 3. Furthermore, for each of these groups $G$ and each finite abelian group $A$ with $(|G|,|A|)=1$, the group $G\times A$ also provides an example.
\bigskip

Next we are interested in characterizing positive integers $d$ such that there exists a finite group $G$ with $f(G)=\frac{1}{d}$\,.

\begin{theorem}
Let $d$ be a positive integer. Write $d=p_1^{n_1}\cdots p_k^{n_k}$, where the $p_i$ are distinct primes, $i\in\{1,...,k\}$. The following hold:
\begin{itemize}
\item[{\rm 1.}] If $n_1=\dots=n_k=0$ or $n_i\geq 2$ for all $i=1,\dots,k$, then there exists a finite group $G$ with $f(G)=\frac{1}{d}$\,.
\item[{\rm 2.}] If $n_1=\dots=n_k=1$, then $f(G)=\frac{1}{d}$ for no finite group $G$.
\end{itemize}
\end{theorem}
\bigskip

By Theorem 1.2 it follows that the largest values of the function $f$ are $1$, $\frac{1}{4}$\,, $\frac{1}{8}$\,, $\frac{1}{9}$\,, ..., and so on. We also observe that for certain finite groups $G$ we have $f(G)=\frac{1}{d}$\,, where $d=p_1^{n_1}\cdots p_k^{n_k}$ satisfies the property that there exist $i,j\in\{1,\dots,k\}$ with $n_i\geq 2$ and $n_j=1$. An example is the group $G=S_4$, for which we have $f(G)=\frac{1}{24}=\frac{1}{2^3\cdot 3}$\,. This leads to the following natural problem.
\bigskip

\noindent{\bf Open problem.} Determine all positive integers $d$ such that there exists a finite group $G$ such that $f(G)=\frac{1}{d}$\,.
\bigskip

Most of our notation is standard and will not be repeated here. Basic definitions and results on groups can be found in \cite{6}.

\section{Proofs of the main results}

First of all, we recall some criteria for supersolvability of finite groups (see Lemmas 3.4 and 3.12 of \cite{1}).

\begin{lemma}
Given a finite group $G$, we have:
\begin{itemize}
\item[{\rm 1.}] $G$ is supersolvable if and only if $G/Z(G)$ is supersolvable.
\item[{\rm 2.}] If $G$ has an abelian subgroup of index $2$, then $G$ is supersolvable.
\end{itemize}      
\end{lemma}

Next, we make the following remark.

\begin{lemma}
Let $G$ be a finite non-abelian group. Then 
\begin{equation}
f(G)\leq\dd\frac{1}{p^2}\,,\nonumber
\end{equation}where $p$ is the smallest prime divisor of the order of $G$. Moreover, the equality holds if and only if $G$ has a self-centralizing subgroup of index $p$.
\end{lemma}

\begin{proof}
Since $G$ is not abelian, we have $M(G)\neq G$ and $\frac{1}{[G:Z(G)]}\leq\frac{1}{p^2}$\,. If $M<G$, then $[G:M]\geq p$, $[G:M(G)]\geq p$, and (2), a),  implies $f(G)\leq\frac{1}{p^2}$\,. If $M=G$, then $M(G)=Z(G)$, and again $f(G)\leq\frac{1}{p^2}$ by (2), a). In any case, $f(G)\leq\frac{1}{p^2}$\,, as required.

The equality holds if and only if either $M(G)=Z(G)$, $M=G$ and $[G:Z(G)]=p^2$, or $M(G)=M$ is of index $p$. In both cases $G$ has a self-centralizing subgroup of index $p$. 
\end{proof}

We are now able to prove our main results.

\bigskip\noindent{\bf Proof of Theorem 1.1.} 
\begin{itemize}
\item[{\rm 1.}] This follows directly from Lemma 2.2.
\item[{\rm 2.}] Assume that $f(G)>\frac{1}{9}$\,. Then $[G:M(G)][G:M]<9$ by (2), a). Since $[G:M]$ divides $[G:M(G)]$, we have $[G:M]\in\{1,2\}$.

\hspace{5mm}If $[G:M]=1$, that is $M=G$ and $M(G)=Z(G)$, then $[G:Z(G)]<9$. Since every group of order at most $8$ is supersolvable, we obtain that $G/Z(G)$ is supersolvable and the conclusion follows from item 1 of Lemma 2.1.

\hspace{5mm}If $[G:M]=2$, then (2), b), implies that $f(M)>\frac{4}{9}>\frac{1}{4}$\,, and so $M$ is abelian by the above item 1. Now the conclusion follows from item 2 of Lemma 2.1.

Assume now that $f(G)=\frac{1}{9}$\,. Arguing as in the proof of Lemma 2.2, we get that $G$ has a self-centralizing subgroup of index $3$.
\item[{\rm 3.}] We proceed by induction on $|G|$. Clearly, the statement holds for $|G|=1$. Assume that it holds for any finite group of order less than $|G|$. Since $f(G)>\frac{1}{60}$\,, by (2), b), it follows that $f(M)>\frac{1}{60}$\,.\newpage 
    
\hspace{5mm}If $M\neq G$, then $|M|<|G|$ and so $M$ is solvable by the inductive hypothesis. On the other hand, since $[G:M(G)][G:M]<60$ we get $[G:M]<60$ and hence $G/M$ is solvable. These imply that $G$ itself is solvable. 

\hspace{5mm}If $M=G$, then $M(G)=Z(G)$ and the inequality 
\begin{equation}
[G:M(G)][G:M]<60\nonumber 
\end{equation}yields to $[G:Z(G)]<60$, that is $G/Z(G)$ is solvable. Since $Z(G)$ is abelian, it is also solvable. Thus $G$ is solvable, too.

Assume now that $f(G)=\frac{1}{60}$\,. Then $[G:M(G)][G:M]=60$ and since $[G:M]$ divides $[G:M(G)]$, we distinguish the following two cases:

\hspace{5mm}a) $[G:M]=2$\\
Then $f(M)=\frac{1}{15}$ by (2), b), which implies that $[M:Z(M)]=15$. Thus $M/Z(M)$ is cyclic and consequently $M$ is abelian, a contradiction.

\hspace{5mm}b) $[G:M]=1$\\
Then $M=G$, $M(G)=Z(G)$ and $[G:Z(G)]=60$. If $G/Z(G)\not\cong A_5$, then it is easy to see that $G/Z(G)$ has a cyclic subgroup of order $15$, say $H/Z(G)$. It follows that $H/Z(H)$ is cyclic and therefore $H$ is abelian. This leads to
\begin{equation}
\dd\frac{1}{60}=f(G)=\dd\frac{m(G)}{|G|^2}\geq\dd\frac{m_G(H)}{|G|^2}\geq\dd\frac{|H|^2}{|G|^2}=\dd\frac{1}{16}\,,\nonumber
\end{equation}a contradiction. Thus $G/Z(G)\cong A_5$, as desired.\qed
\end{itemize}

\bigskip\noindent{\bf Proof of Theorem 1.2.} 
\begin{itemize}
\item[{\rm 1.}] Let $p$ be a prime and $n\in\mathbb{N}$, $n\geq 2$. We will prove that there exists a finite $p$-group $G_p^n$ such that $f(G_p^n)=\frac{1}{p^n}$\,. Generalizing Proposition 2.3 of \cite{2}, we construct
\begin{equation}
G_p^n=\langle x_1, x_2, \dots, x_n, x_{1,2}, x_{1,3}, \dots, x_{n-1,n}\rangle,\nonumber
\end{equation}where
\begin{itemize}
\item[-] $x_1^p=x_2^p=\dots=x_n^p=x_{1,2}^p=x_{1,3}^p=\dots=x_{n-1,n}^p=1$,
\item[-] $[x_i,x_j]=x_{i,j}, \forall\, 1\leq i<j\leq n$,
\item[-] all other commutators between the generators equal $1$.
\end{itemize}Then $|G_p^n|=p^{n+\frac{n(n-1)}{2}}$ and $m(G_p^n)=p^{n^2}$, implying that $f(G_p^n)=\frac{1}{p^n}$\,, as desired.

If $n_i\geq 2$ for all $i=1,\dots,k$, then it suffices to take $G=G_{p_1}^{n_1}\times\cdots\times G_{p_k}^{n_k}$ and then to apply (1).
\item[{\rm 2.}] Assume that there is a finite group $G$ with $f(G)=\frac{1}{p_1\cdots p_k}$\,. Then the conditions $[G:M(G)][G:M]=p_1\cdots p_k$ and $[G:M] \mid [G:M(G)]$ imply that $M=G$ and $M(G)=Z(G)$. It follows that $G/Z(G)$ is a ZM-group, say $G/Z(G)\cong ZM(m,n,r)$. We can suppose that $m<n$ (the case $n<m$ is similar). Similarly with the proof of item 3 of Theorem 1.2, we obtain that $G$ has an abelian subgroup $H$ of index $m$. Then
\begin{equation}
\dd\frac{1}{mn}=\dd\frac{1}{p_1\cdots p_k}=f(G)=\dd\frac{m(G)}{|G|^2}\geq\dd\frac{m_G(H)}{|G|^2}\geq\dd\frac{|H|^2}{|G|^2}=\dd\frac{1}{m^2}\,,\nonumber
\end{equation}that is $m\geq n$, a contradiction. This completes the proof.\qed
\end{itemize}

Finally, we note that for an extraspecial $p$-group $P$ of order $p^{2n+1}$ we have $f(P)=\frac{1}{p^{2n}}$ (see e.g. \cite{4}, Example 2.8, or \cite{7}, Theorem 4.3.4). Then for $d=p_1^{2n_1}\cdots p_k^{2n_k}$ with $n_i\in\mathbb{N}^*$, $\forall\, i=1,\dots,k$, we can also take $G=G_1\times\cdots\times G_k$, where $G_i$ is an extraspecial $p_i$-group of order $p_i^{2n_i+1}$, $i=1,\dots,k$, obtaining $f(G)=\frac{1}{d}$\,.

\bigskip\noindent{\bf Acknowledgements.} The author is grateful to the reviewers for their remarks which improve the previous version of the paper.

\vspace*{5ex}\small

\hfill
\begin{minipage}[t]{5cm}
Marius T\u arn\u auceanu \\
Faculty of  Mathematics \\
``Al.I. Cuza'' University \\
Ia\c si, Romania \\
e-mail: {\tt tarnauc@uaic.ro}
\end{minipage}

\end{document}